\documentclass[12pt]{amsart}
\usepackage[margin=1in]{geometry}
\usepackage[leqno]{amsmath}
\usepackage{amssymb}
\usepackage{amsthm}
\usepackage{nicematrix}
\usepackage{anyfontsize}
\usepackage{hyperref}
\usepackage{ytableau}
\usepackage{youngtab}

\newtheorem{theorem}{Theorem}[section]
\newtheorem{corollary}[theorem]{Corollary}
\newtheorem{proposition}[theorem]{Proposition}
\newtheorem{lemma}[theorem]{Lemma}

\newtheorem*{theorem*}{Theorem}

\theoremstyle{remark}
\newtheorem{remark}[theorem]{Remark}
\theoremstyle{definition}

\newtheorem{example}[theorem]{Example}
\newtheorem*{definition*}{Definition}

\newcommand{\C}{\mathbb{C}}
\newcommand{\Z}{\mathbb{Z}}
\newcommand{\Q}{\mathbb{Q}}
\newcommand{\cF}{\mathcal{F}}
\newcommand{\of}{\circ}
\newcommand{\ox}{\otimes}
\DeclareMathOperator{\Pf}{Pf}

\newcommand{\bibDate}[1]{}

\begin{document}
\title{Plethysm Stability of Schur's $Q$-functions}

\author{John~Graf}
\address{Department of Mathematics, North Carolina State University, Raleigh, NC 27695, USA}
\email[John~Graf (Corresponding Author)]{jrgraf@ncsu.edu}

\author{Naihuan~Jing}
\address{Department of Mathematics, North Carolina State University, Raleigh, NC 27695, USA}
\email[Naihuan~Jing]{jing@ncsu.edu}
\thanks{Supported in part by Simons Foundation MP-TSM-00002518 and NSFC 12171303}

\date{March 27, 2025} 
    
    \begin{abstract}
        Schur functions has been shown to satisfy certain plethysm stability properties and recurrence relations. In this paper, use vertex operator methods to study analogous stability properties of Schur's $Q$-functions. Although the two functions have similar stability properties, we find a special case where the plethysm of Schur's $Q$-functions exhibits linear increase.
    \end{abstract}

    \maketitle
    
    \tableofcontents
    
    \section{Introduction}

    Schur functions $S_\lambda$ form a basis of the ring of symmetric functions $\Lambda$ and correspond to irreducible representations of the symmetric group under the characteristic map \cite{macdonald:1995}. Similarly, Schur's $Q$-functions $Q_\lambda$ serve a similar role as a basis of the subring $\Gamma\subset\Lambda$ that correspond to irreducible projective representations of the double covering group of the symmetric group (also called the irreducible spin representation). Indeed, both bases have many similar properties. The Jacobi-Trudi formula provides a determinantal method for computing Schur functions, which shows that any irreducible representation of the symmetric group can be composed from the basic representation or the sign representation, and it allows one to compute $S_\lambda$ for compositions with negative parts. For the spin representation, the analogous Pfaffian formula also shows that any irreducible spin representation can be composed from the basic spin representation \cite{schur:1911}. In a previous paper \cite{graf-jing:2024}, we extended the Pfaffian formula for Schur's $Q$-functions to similarly allow compositions $\lambda$ with negative parts.

    Let $\pi_{\lambda}$ be the $GL$-representation associated with partition $\lambda$. Littlewood studied the plethysm $\pi_{\lambda}\circ\pi_{\mu}$, arising from the composition $\pi_{\mu}(\pi_{\lambda})$ in the category of $GL$-representations \cite{Lit, Littlewood}, and correspondingly the plethysm of any two symmetric functions can be computed via the Frobenius characteristic map (cf. \cite[p.~160]{macdonald:1995}). It is a main problem to find information on the coefficients $a^{\nu}_{\lambda,\mu}$ in the plethysm decomposition $\pi_{\lambda}\circ\pi_{\mu}=\oplus_{\nu}a^{\nu}_{\lambda,\mu}\pi_{\nu}$. However, these coefficients remain mysterious \cite{colmenarejo2024} even for small rank \cite{Paget-Wildon}. Some basics of plethysm are explained in \cite{loehr-remmel} and
    \cite{FJK}. Also see \cite{colmenarejo2024} for some recent developments.
    
    An important property for plethysm of Schur functions is the stability, which asserts that several sequences of plethysm coefficients have been shown to stabilize \cite{Wei, carre-thibon:1992}. Carr\'e and Thibon \cite{carre-thibon:1992} used methods arising from vertex operators to prove that the sequences 
    \begin{align*}
        (S_\lambda\of S_{p\mu},S_{s\nu})&\qquad\text{for }p\in\Z,\text{ }s=|\lambda|(|\mu|+p)-|\nu|,\\
        (S_{p\lambda}\of S_{\mu},S_{s\nu})&\qquad\text{for }p\in\Z,\text{ }s=(|\lambda|+p)|\mu|-|\nu|
    \end{align*}
    stabilize for large enough $p$, where $p\lambda=(p,\lambda_1,\ldots,\lambda_n)$. Additionally, Brion \cite{brion:1993} proved stability theorems using geometric methods, and Colmenarejo \cite{colmenarejo:2017} used combinatorial methods. Although these stability theorems have been widely studied for Schur functions, the analogous properties had not been established for Schur's $Q$-functions. 

    Hence, in this paper we study some analogous stability properties of Schur's $Q$-functions via a careful adaption of Carr\'e and Thibon's vertex operator methods for the case of twisted vertex operators. We find that plethysm of Schur's Q-functions almost enjoys the stability except in a special case! The idea of our current treatment can be traced back to two well-known constructions of the simplest affine Lie algebra $\widehat{\mathfrak sl}(2)$ by the homogeneous untwised vertex operators and principal twisted vertex operators (cf.\cite{FLM}) and their applications to symmetric functions \cite{jing:1991a}. In particular, we show that the sequences
    \begin{align*}
        (Q_\lambda\of Q_{p\mu},Q_{s\nu})&\qquad\text{for }p\in\Z,\text{ }s=|\lambda|(|\mu|+p)-|\nu|,\\
        (Q_{p\lambda}\of Q_{\mu},Q_{s\nu})&\qquad\text{for }p\in\Z,\text{ }s=(|\lambda|+p)|\mu|-|\nu|,~\ell(\mu)>1
    \end{align*}
    stabilize for large enough $p$. Meanwhile, the sequence 
    \[(Q_{p\lambda}\of Q_{(m)},Q_{s\nu})\qquad\text{for }p\in\Z,\text{ }s=(|\lambda|+p)m-|\nu|\]
    increases linearly for large enough $p$. 

    Furthermore, Carr\'e and Thibon \cite{carre-thibon:1992} used their vertex operator methods to prove a recurrence formula for the plethysm of Schur functions. As special cases, this formula reduces to different recurrence formulas provided by Butler and King \cite{butler-king:1973} and by Murnaghan \cite{murnaghan:1954}. Consequently, we also derive analogous recurrence formulas for Schur's $Q$-functions.


    The Hall-Littlewood functions $Q_\lambda(A;t)$ are a generalization of both the Schur and Schur's $Q$-functions, and they also have a vertex operator realization \cite{jing:1991b}. So, it is desirable to generalize these results. However, some formulas of Hall-Littlewood functions have not been established, and in particular we bring attention to the open problem of generalizing the Jacobi-Trudi formula to skew Hall-Littlewood functions $Q_{\lambda/\mu}(A;t)$.

    \section{Preliminaries}

    \subsection{Partitions}

    A \emph{composition} $\lambda=(\lambda_1,\ldots,\lambda_n)\in\Z^n$ is a sequence of integers, and it is a \emph{partition} if its \emph{parts} $\lambda_1,\lambda_2,\ldots$ are nonnegative and weakly decreasing. The \emph{length} $\ell(\lambda)$  of $\lambda$ is the number of nonzero parts. The \emph{weight} $|\lambda|$ of $\lambda$ is defined to be the sum of its parts. A partition is \emph{strict} if it has no repeated nonzero parts. 
    
    We define some operations to append or remove parts to a composition $\lambda$. First, for any integer $p\in\Z$ we define
        \[p\lambda:=(p,\lambda_1,\ldots,\lambda_n).\]
    Next, for any integer $i\in\{1,2,\ldots,n\}$, we define
    \[\lambda\setminus\{\lambda_i\}:=(\lambda_1,\ldots,\lambda_{i-1},\lambda_{i+1},\ldots,\lambda_n).\]

    \subsection{Schur's $Q$-functions}

    We will provide a brief overview of Schur's $Q$-functions \cite{schur:1911, macdonald:1995}. A more detailed treatment can be found in \cite{graf-jing:2024}.
	
    An \emph{alphabet} $A=\{a_1,a_2,\ldots\}$ is a set of variables, which may be infinite. The functions $q_n(A)$ of the alphabet $A$ are defined by the generating series
    \[\kappa_z(A):=\prod_{a\in A}\frac{1+za}{1-za}=\sum_{n\in\Z}q_n(A)z^n.\]
    Note that $q_0=1$, and that $q_{n}=0$ for $n<0$. For any integers $r,s\in\Z$, we define
    \begin{equation}\label{Q_(r,s)}
        Q_{(r,s)}(A):=q_r(A)q_s(A)+2\sum_{i=1}^s(-1)^iq_{r+i}(A)q_{s-i}(A).
    \end{equation}
    For any composition $\lambda\in\Z^n$, we define \emph{Schur's $Q$-function} $Q_\lambda(A)$ to be the Pfaffian $\Pf$ of the matrix $M(\lambda)$ with entries 
    \[M(\lambda)_{ij}=
        \begin{cases}
            Q_{(\lambda_i,\lambda_j)}&\text{if }i<j,\\
            0&\text{if }i=j,\\
            -Q_{(\lambda_j,\lambda_i)}&\text{if }i>j,
        \end{cases}\]
    where we use $\lambda0:=(\lambda_1,\ldots,\lambda_n,0)$ if $n$ is odd, and where $(\Pf M(\lambda))^2=\det M$. For any partitions $\lambda\in\Z^n$ and $\mu\in\Z^m$, we define the \emph{skew Schur's $Q$-function} $Q_{\lambda/\mu}(A)$ by
    \begin{equation}\label{Q_lambda/mu formula}
        Q_{\lambda/\mu}(A):=
        \begin{cases}
            \Pf M(\lambda,\mu)&\text{if }n+m\text{ is even},\\
            \Pf M(\lambda0,\mu)&\text{if }n+m\text{ is odd},
        \end{cases}
    \end{equation}
    where $M(\lambda,\mu)$ denotes the matrix
    \begin{equation*}\label{M(lambda,mu)}
        M(\lambda,\mu):=
        \begin{pmatrix}
            M(\lambda)&N(\lambda,\mu)\\
            -N(\lambda,\mu)^t&0
        \end{pmatrix},
    \end{equation*}
    and where $N(\lambda,\mu)$ is the $n\times m$ matrix
    \begin{equation*}
        N(\lambda,\mu):=
        \begin{pmatrix}
            q_{\lambda_1-\mu_m}&\cdots&q_{\lambda_1-\mu_1}\\
            \vdots&&\vdots\\
            q_{\lambda_n-\mu_m}&\cdots&q_{\lambda_n-\mu_1}
        \end{pmatrix}.
    \end{equation*}
    
    Henceforth, omitting the variables of functions will be understood to mean we are using the alphabet $A$. That is, we write $q_n,Q_\lambda,\kappa_z,$ etc. to denote $q_n(A),Q_\lambda(A),\kappa_z(A),$ etc., respectively.

    We define the ring $\Gamma:=\Q<q_1,q_2,q_3,\ldots>$ to be the ring spanned by the $q_n$. It follows from the definition of $\kappa_z$ that the odd $q_n$'s are algebraically independent, and $\Gamma=\Z[q_1,q_3,q_5,\ldots]$. The $Q_\lambda$ form a basis for $\Gamma$, where $\lambda$ ranges over strict partitions. Note that $\Gamma$ is a subring of the ring $\Lambda$ of symmetric functions. See Appendix \ref{appendix:schur} for details about $\Lambda$ and Schur functions, and see Appendix \ref{appendix:combinatorial} for details about alternative constructions of Schur's $Q$-functions.




    \subsection{Inner Product and Adjoints}
    We define an inner product on the ring $\Gamma$ by setting
    \begin{equation}(Q_\lambda,Q_\mu)=2^{\ell(\lambda)}\delta_{\lambda\mu}
    \end{equation}
    for strict partitions $\lambda$ and $\mu$. We extend the inner product to $\Gamma\ox\C[z]$ by $\C[z]$-linearity, where $z$ is an indeterminate.

    For any partitions $\lambda$ and $\mu$, the skew Schur's $Q$-function $Q_{\lambda/\mu}$ satisfies the identity
    \begin{equation}\label{e:skew}
    (Q_{\lambda/\mu},F)=(Q_\lambda,2^{-\ell(\mu)}Q_\mu F)
    \end{equation}
    for all $F\in\Gamma$.
    
    For any function $F\in\Gamma$, we let $F^\perp$ denote the adjoint of multiplication by $F$ with respect to $(\cdot,\cdot)$,
    \begin{equation}
      (F^\perp Q_\lambda,Q_\mu)=(Q_\lambda,FQ_\mu).  
    \end{equation}
    If $F=\sum_n F_nz^n$ is an infinite series, we denote $F^\perp:=\sum_n z^nF_n^\perp$.
    
    
    
    

    \subsection{Plethysm}

    \subsubsection*{$\lambda$-ring formulism of plethysm}
    
    We use the $\lambda$-ring definition of plethysm from \cite{lascoux:2003} to define plethysm on $\Gamma$.
    
    Let $X$ be any alphabet, which may contain $A$. We associate every element $F\in\Gamma(A)$ with an operator, also denoted $F$, that acts on elements of the polynomial ring $\C[X]$. Let $x^\mu:=x_1^{\mu_1}x_2^{\mu_2}\cdots$, and suppose $P=\sum_\mu c_\mu x^\mu\in\C[X]$. We first define $q_n(P)$ by the generating series
    \begin{equation}\label{eqn:kappa_z(P)}
        \kappa_z(P)=\prod_\mu\left(\frac{1+zx^\mu}{1-zx^\mu}\right)^{c_\mu}=\sum_{n\in\Z}q_n(P)z^n.
    \end{equation}
    Next, note that any element $F(A)\in\Gamma(A)$ can be written as a polynomial in the $q_n(A)$; $F(A)=\cF(q_1(A),q_2(A),\ldots)$. So, we define plethysm of $F$ on $P$ as
    \[F(P):=\cF(q_1(P),q_2(P),\ldots).\]
    
    We denote the plethysm of $F,G\in\Gamma$ as either $F(G(A))$ or $F\of G(A)$. See Appendix \ref{appendix:schur} for the connection of our definition of plethysm on $\Gamma$ to plethysm on the larger ring $\Lambda$.



    \subsubsection*{Plethystic notation}

    Since we have $F(A)=F(a_1+a_2+\cdots)$ for all $F\in\Gamma$, we can identify an alphabet with the formal sum of its elements. For two alphabets $A$ and $B$, the sum $A+B$ is the \emph{disjoint} union of $A$ and $B$, so that
    \[kA=\underbrace{A+\cdots+A}_{k\text{ terms}}\]
    for any positive integer $k$. Equivalently, we have $\kappa_z(kA)=(\kappa_z(A))^k$ for positive integers $k$, which we extend by defining
    \[\kappa_z(kA):=(\kappa_z(A))^k,\qquad\text{for all }k\in\C.\]
    In particular, we have
    \begin{equation}\label{eqn:kappa_z(A+B)}
        \kappa_z(A\pm B)=\kappa_z(A)\kappa_z(B)^{\pm1}.
    \end{equation}
    Since $\kappa_z(0)=1$, it follows that $A-A=0$. For example, we have $F(A+B-B)=F(A)$ for all $F\in\Gamma$. Additionally, we define
    \[AB:=\sum_{\substack{a\in A\\b\in B}}ab,\]
    and so in particular $zA=za_1+za_2+\cdots$.

    \subsubsection*{Power sum computations}

    Computing specific examples of plethysm is often easiest using the power sum basis. For all $n\geq1$ we define the \emph{power sum function} $p_n$ by
    \begin{equation*}
        p_n(A):=\sum_{a\in A}a^n.
    \end{equation*}
    We note that $\Gamma$ is generated by the odd power sums, $\Gamma=\Q[p_1,p_3,p_5,\ldots]$. For a partition $\mu$, we define $p_\mu:=p_{\mu_1}p_{\mu_2}\cdots$, and
    \begin{align*}
        z_\mu&:=\prod_{i\geq1}i^{m_i(\mu)}\cdot m_i(\mu)!,
    \end{align*}
    where $m_i(\mu)=\#\{j\mid \mu_j=i\}$ is the number of parts of $\mu$ equal to $i$. In terms of the power sum basis, for $n\geq1$ we have \cite[p.~260]{macdonald:1995}
    \begin{equation}\label{eqn:power sums}
        q_n=\sum_{\mu\text{ odd}}z_{\mu}^{-1}2^{\ell(\mu)}p_\mu,
    \end{equation}
    where a partition $\mu$ is \emph{odd} if each of its parts are odd. To compute $F\of G$ for $F,G\in\Gamma$, we can write $F$ and $G$ in the power sum basis using (\ref{eqn:power sums}), then compute the plethysm using the method described in \cite[p.~168]{loehr-remmel}.



    Let $\omega:\Lambda\to\Lambda$ be the involution $\omega(h_n)=e_n$ exchanging the homogeneous and elementary symmetric functions. Since $\omega(p_n)=(-1)^{n-1}p_n$, we have that $\omega(p_{2n+1})=p_{2n+1}$, and hence all elements of $\Gamma$ are invariant under $\omega$. In general, for a homogeneous symmetric function $F\in\Lambda$ of degree $n$ we have $F(-A)=(-1)^n(\omega F)(A)$ \cite[p.~137]{macdonald:1995}. Thus, for all homogeneous $F\in\Gamma$ of degree $n$ we have
    \begin{equation}\label{eqn:negative alphabet}
        F(-A)=(-1)^nF(A).
    \end{equation}

    \subsubsection*{Plethysm properties}
    
    Before we continue, we state some important identities of plethysm.
    
    \begin{proposition} [Sum Rule]\label{sum rule}
        For any partition $\lambda$ and any two alphabets $A$ and $B$, we have
    \begin{equation}
     Q_\lambda(A+B)=\sum_\mu Q_{\lambda/\mu}(A)Q_\mu(B),   
    \end{equation}
    where the sum is over all partitions $\mu$.
    \end{proposition}
    \begin{proof}
        This is a restatement of equation (5.5) from \cite[p.~228]{macdonald:1995}, setting $t=-1$.
    \end{proof}
    More generally, we have the following result.
    \begin{corollary}[Skew Sum Rule]
        For any partitions $\lambda,\mu$ and any two alphabets $A$ and $B$, we have
        \begin{equation}
         Q_{\lambda/\mu}(A+B)=\sum_\nu Q_{\lambda/\nu}(A)Q_{\nu/\mu}(B),   
        \end{equation}
        where the sum is over partitions $\nu$.
    \end{corollary}
    \begin{proof}
        Following the analogous proof from \cite[p.~72]{macdonald:1995} for Schur functions, we have from Proposition \ref{sum rule} that
        \begin{align*}
            \sum_\mu Q_{\lambda/\mu}(A+B)Q_\mu(C)&=Q_\lambda((A+B)+C)\\
            &=Q_\lambda(A+(B+C))\\
            &=\sum_\nu Q_{\lambda/\nu}(A)Q_\nu(B+C)\\
            &=\sum_{\mu,\nu}Q_{\lambda/\nu}(A)Q_{\nu/\mu}(B)Q_\mu(C).
        \end{align*}
        Finally, we equate the coefficients of $Q_\mu(C)$ in the beginning and end of the above chain of equalities.
    \end{proof}
    
    If we use the alphabets $A=\{z\}$, $A=\{-z\}$, or $A=\{2z\}$, then we get the following results.
    
    \begin{proposition}\label{Q_lambda(z)}
        We have
        \begin{align}
            q_0(z)&=1,\\
            q_n(z)&=2z^n\qquad(n\geq1), \label{q_r(z)} \\
            q_n(-z)&=2(-z)^n\quad(n\geq1), \label{q_r(-z)}\\
            q_n(2z)&=4nz^n\qquad(n\geq1). \label{q_r(2z)}
        \end{align}
        Moreover, if $\lambda$ is a partition such that $\ell(\lambda)>1$, then $Q_\lambda(z)=0$. 
    \end{proposition}
    \begin{proof}
        Let $w$ be another indeterminate. From the definition of plethysm, we have
        \begin{align*}
            \kappa_w(z)&=\frac{1+wz}{1-wz}\\
            &=(1+wz)\sum_{n\geq0}(wz)^n\\
            &=\sum_{n\geq0}z^nw^n+\sum_{n\geq1}z^nw^n\\
            &=1+\sum_{n\geq1}(2z^n)w^n
        \end{align*}
        proving the first two equations. Equation (\ref{q_r(-z)}) is an immediate consequence of (\ref{q_r(z)}) and (\ref{eqn:negative alphabet}).

        Now, by the sum rule we have
        \begin{align*}
            q_n(2z)&=\sum_{i=0}^n q_{n-i}(z)q_i(z)\\
            &=2q_n(z)+\sum_{i=1}^{n-1} q_{n-i}(z)q_i(z)\\
            &=4z^n+(n-1)4z^n\\
            &=4nz^n.
        \end{align*}

        For the last result, it suffices to show this for $\ell(\lambda)=2$ since we define $Q_\lambda$ as a polynomial in the functions $Q_{(r,s)}$. So, suppose $\ell(\lambda)=2$, then we see that
        \begin{align*}
            Q_{(r,s)}(z)&=q_r(z)q_s(z)+2\sum_{i=1}^s(-1)^iq_{r+i}(z)q_{s-i}(z)\\
            &=2z^r\cdot2z^s+2\sum_{i=1}^{s-1}(-1)^i2z^{r+i}\cdot2z^{s-i}+2(-1)^s\cdot 2z^{r+s}\cdot 1\\
            &=4z^{r+s}\left(1+(-1)^s+2\sum_{i=1}^{s-1}(-1)^i\right).
        \end{align*}
        If $s$ is even, then in the parentheses we have $1+1+2\cdot (-1)=0$. If $s$ is odd, then we have $1-1+2\cdot0=0$. Therefore, in either case we have $Q_{(r,s)}(z)=0$.
    \end{proof}
    
    \begin{proposition}\label{plethysm zA}
        For any partition $\lambda$, we have
        \begin{equation}
            Q_\lambda(zA)=z^{|\lambda|}Q_\lambda(A).
        \end{equation}
    \end{proposition}
\begin{proof}
        This equation follows from the fact that $zA=za_1+za_2+\cdots$, and the fact that $Q_\lambda(A)$ is a homogeneous polynomial in the $a_i$ of degree $|\lambda|$.
    \end{proof}
    
    \subsection{Vertex operator formalism of Schur Q-functions}

    We recall several identities from \cite{graf-jing:2024}, which will be useful in our consideration.
    First, we calculate the action of $q_r^\perp$ as follows.
    
    \begin{proposition}\label{adjoint}
        For any partition $\lambda$, we have
        \begin{align*}
            q_r^\perp Q_\lambda&=2Q_{\lambda/(r)}\qquad(r\geq1),\\
            q_0^\perp Q_\lambda&=Q_{\lambda}.
        \end{align*}
    \end{proposition}

    Next, appending a negative part to $\lambda$ is, up to a coefficient, the same as removing a part from $\lambda$.
    
    \begin{proposition}\label{Q_(-p)lambda}
        Let $p>0$, be a positive integer and let $\lambda$ be a strict partition, then
        \[Q_{(-p)\lambda}=
            \begin{cases}
                (-1)^{p+i+1}2Q_{\lambda\setminus\{\lambda_i\}}&\text{if }p=\lambda_i\text{ for some i},\\
                0&\text{otherwise}.
            \end{cases}
        \]
    \end{proposition}

    If $\lambda$ is not a partition, then the following result states that $Q_\lambda=kQ_\tau$ for some constant $k$ and partition $\tau$.

    \begin{proposition}\label{B_iQ_lambda}
        Let $\lambda\in\Z^n$ be a composition, and let $B_i$ act on $\lambda$ by swapping its $i$th and $(i+1)$st parts. Then
        \[Q_{B_i\lambda}=
        \begin{cases}
            Q_\lambda&\text{if }\lambda_i=\lambda_{i+1},\\
            -Q_\lambda&\text{if }\lambda_i+\lambda_{i+1}\neq0,\\
            (-1)^{\lambda_i}2Q_{\lambda\setminus\{\lambda_i,\lambda_{i+1}\}}&\text{if }\lambda_i+\lambda_{i+1}=0\text{ and }\lambda_i>0,\\
            0&\text{if }\lambda_i+\lambda_{i+1}=0\text{ and }\lambda_i<0.
        \end{cases}\]
    \end{proposition}

    Consequently, we have that the vertex operator $\kappa_z\cdot \kappa_{-1/z}^\perp$ acts on $Q_\lambda$ in the following way.
    
    \begin{theorem}\label{vertex operator identity}
        Let $\lambda$ be a partition, then we have
        \[\kappa_z\cdot \kappa_{-1/z}^\perp Q_\lambda=\sum_{p\in\Z}Q_{p\lambda}z^p.\]
    \end{theorem}

    Lastly, we have that a specific type of skew Schur's $Q$-function is equivalent to the multiplication of two Schur's $Q$-functions.
    
    \begin{proposition}\label{p>t+r}
        For all partitions $\lambda\in\Z^n$ and integers $k,p\in\Z$ such that $k\geq0$ and $p>\lambda_1+k$, we have
        \[Q_{p\lambda/(p-k)}=q_kQ_\lambda.\]
    \end{proposition}

    \section{Stability Theorems}
    
    In this section, we will show that the stability theorems of plethysm of Schur functions \cite{carre-thibon:1992} also hold for Schur's $Q$-functions, except in one special case. First, in order to use the vertex operator identity, we will show how $\kappa_z^\perp$ and $\kappa_{-z}^\perp$ act on Schur's $Q$-functions.
    
    \begin{lemma}\label{F(A+z)}
        Let $F\in\Gamma\ox\C[z]$, then we have 
        \begin{align*}
            \kappa_z^\perp F(A)&=F(A+z),\\
            \kappa_{-z}^\perp F(A)&=F(A-z).
        \end{align*}
    \end{lemma}
    
    \begin{proof}
        We proceed similarly to \cite[p.~95]{macdonald:1995}. First, we compute the action of $\kappa_z^\perp$ on a basis element $Q_\lambda$. So, we have
        \begin{align*}
            \kappa_z^\perp Q_\lambda(A)&=\sum_{r\geq0}z^r \cdot q_r^\perp Q_\lambda(A)\\
            &=q_0^\perp Q_\lambda(A)q_0(z)+\sum_{r\geq1}q_r^\perp Q_\lambda(A)\cdot\frac{1}{2}q_r(z)
        \end{align*}
        since we see that $1=q_0(z)$ and $z^r=\frac{1}{2}q_r(z)$ ($r\geq1$) by \eqref{q_r(z)}. Then, we have that $q_0^\perp Q_\lambda=Q_\lambda$ and $\frac{1}{2}q_r^\perp Q_\lambda=Q_{\lambda/(r)}$ ($r\geq1$) by Proposition \ref{adjoint}, and so we get
        \[Q_{\lambda/(0)}(A)q_0(z)+\sum_{r\geq1}Q_{\lambda/(r)}(A)q_r(z)=\sum_{r\geq0}Q_{\lambda/(r)}(A)q_r(z).\]
        By Proposition \ref{plethysm zA}, we have that $Q_\mu(z)=0$ if $\ell(\mu)>1$, and so the sum is
        \[\sum_{\mu}Q_{\lambda/\mu}(A)Q_\mu(z).\]
        Finally, by the sum rule, we see that we have $Q_\lambda(A+z)$, as desired. The general case follows since the $Q_\lambda$ form a basis of $\Gamma$. The second identity can be proven similarly using \eqref{q_r(-z)}.
    \end{proof}

    We will write our sequences as the coefficients of Laurent series. We say that the \emph{degree} of a Laurent polynomial $L(Z)$ is the highest power of $Z$ that appears with a nonzero coefficient. The following identity will allow us to separate a particular product into a sum.

    \begin{lemma}\label{separate}
        Let $n\in\Z$ and let $H(Z)\in\C[Z,Z^{-1}]$ be a Laurent polynomial. Then
        \begin{equation*}
            \frac{H(Z)}{(1-Z)^n}=L(Z)+\sum_{i=1}^n\frac{c_i}{(1-Z)^i},
        \end{equation*}
        where $c_i\in\C$ for all $i$, and $L(Z)$ is a Laurent polynomial. If $H(Z)\neq0$, then $L(Z)$ has degree at most $\max(\deg(H)-n,0)$.
    \end{lemma}

    \begin{proof}
        When $H(Z)=0$ or $n=0$, there is nothing to prove. When $n<0$ the sum is empty, and we see that $\deg(L)=\deg(H)-n$. Now, assume $n\geq1$. We proceed by induction on $n$.

        First, suppose $n=1$. We can write
        \[H(Z)=\sum_{i\in I} b_iZ^i,\]
        where $i$ ranges over some nonempty, finite set of integers $I\subset\Z$, and each $b_i\neq0$. First assume that $\# I=1$, and so $H(Z)=b_kZ^k$. In this case we have
        \begin{align*}
            \frac{H(Z)}{1-Z}&=b_kZ^k\sum_{p\geq0}Z^p\\
            &=b_k\sum_{p\geq k}Z^p\\
            &=L(Z)+\frac{b_k}{1-Z},
        \end{align*}
        where
        \begin{equation*}
            L(Z)=
            \begin{cases}
                -b_k\sum_{p=0}^{k-1}Z^p&\text{if }k>0,\\
                0&\text{if }k=0,\\
                b_k\sum_{p=k}^{-1}Z^p&\text{if }k<0.
            \end{cases}
        \end{equation*}
        Additionally, we see that $\deg(L)\leq \max(k-1,0)$. Next, suppose $\#I>1$, and assume that $\max(I)=k$. Let $G(Z)=H(Z)-b_kZ^k$, then by induction on $\#I$ we get
        \begin{align*}
            \frac{H(Z)}{1-Z}&=\frac{G(Z)}{1-Z}+\frac{b_kZ^k}{1-Z}\\
            &=\left(L_1(Z)+\frac{c}{1-Z}\right)+\left(L_2(Z)+\frac{b_k}{1-Z}\right)\\
            &=L(Z)+\frac{(c+b_k)}{1-Z},
        \end{align*}
        where $L(Z)=L_1(Z)+L_2(Z)$, and
        \[\deg(L)\leq\deg(L_1)+\deg(L_2)\leq \max(k-1,0).\]

        Now, suppose $n>1$. Then, by induction on $n$ we have
        \begin{align*}
            \frac{H(Z)}{(1-Z)^n}&=\frac{1}{1-Z}\cdot\frac{H(Z)}{(1-Z)^{n-1}}\\
            &=\frac{1}{1-Z}\left(L(Z)+\sum_{i=1}^{n-1}\frac{c_i}{(1-Z)^i}\right)\\
            &=\frac{L(Z)}{1-Z}+\sum_{i=2}^{n}\frac{c_{i-1}}{(1-Z)^{i}}.
        \end{align*}
        We also have $L(Z)/(1-Z)=P(Z)+c_0/(1-Z)$, and so with reindexing the $c_i$'s we get
        \begin{align*}
            \frac{H(Z)}{(1-Z)^n}&=P(Z)+\sum_{i=1}^n\frac{c_i}{(1-Z)^i}.
        \end{align*}
        Finally, we see that
        \[\deg(P)\leq \max(\deg(L)-1,0)=\max(\deg(H)-(n-1)-1,0).\]
    \end{proof}


    
    

    Now, we use the vertex algebraic method to prove 
    stability results for Schur's $Q$-functions. First, we prove a Schur's $Q$-function analogue of \cite[Theorem~4.1]{carre-thibon:1992}.

    \begin{theorem}\label{stability thm 1}
        Let $\lambda,\mu,\nu$ be partitions, and let $r$ be the greatest integer such that $|\lambda|(|\mu|+r)\leq|\nu|$, then
        \[\sum_{p,s\in\Z}\left(Q_\lambda\of Q_{p\mu},Q_{s\nu}\right)z^s=L(z)+\frac{cz^k}{1-z^{|\lambda|}},\]
        where $c,k\in\Z$, and $L(z)$ is a Laurent polynomial of degree at most $|\lambda|(\mu_1+r)$.
    \end{theorem}
    
    \begin{proof}
        First, we let
        \[f(z)=\sum_{p,s\in\Z}\left(Q_\lambda\of Q_{p\mu},Q_{s\nu}\right)z^s.\]
        By linearity of the inner product, we have
        \[f(z)=\sum_{p\in\Z}\left(Q_\lambda\of Q_{p\mu},\sum_{s\in\Z}z^sQ_{s\nu}\right).\]
        We apply Theorem \ref{vertex operator identity} to the RHS of the inner product to get
        \[f(z)=\sum_{p\in\Z}\left(Q_\lambda\of Q_{p\mu},\kappa_z\cdot \kappa_{-1/z}^\perp Q_\nu\right).\]
        Next, from the definition of adjoint we have
        \[f(z)=\sum_{p\in\Z}\left(\kappa_z^\perp Q_\lambda\of Q_{p\mu},\kappa_{-1/z}^\perp Q_\nu\right),\]
        and so by Lemma \ref{F(A+z)} we get
        \[f(z)=\sum_{p\in\Z}\left(Q_\lambda\of Q_{p\mu}(A+z),Q_\nu(A-1/z)\right).\]
        
        Now, we will write $f(z)=R(z)+T(z)$, where
        \[R(z)=\sum_{p\leq \mu_1+r}\left(Q_\lambda\of Q_{p\mu}(A+z),Q_\nu(A-1/z)\right),\]
        and
        \[T(z)=\sum_{p>\mu_1+r}\left(Q_\lambda\of Q_{p\mu}(A+z),Q_\nu(A-1/z)\right).\]
        It suffices to show that $R(z)$ is a Laurent polynomial, and that $T(z)$ stabilizes. 
        
        First, we expand the plethysm $Q_{p\mu}(A+z)$ using the sum rule and Proposition \ref{Q_lambda(z)} to get
        \begin{align*}
            Q_{p\mu}(A+z)&=\sum_\gamma Q_{p\mu/\gamma}(A)Q_\gamma(z)\\
            &=\sum_i Q_{p\mu/(i)}(A)q_i(z)\\
            &=Q_{p\mu}+2\sum_{i\geq1}Q_{p\mu/(i)}z^i.
        \end{align*}
        Now, notice from Proposition \ref{B_iQ_lambda} that if $p\leq \mu_1+r$, then we have $Q_{p\mu}=Q_{\tau}$ (up to a, possibly-zero, coefficient) for some partition $\tau$ with largest part at most $\mu_1+r$. Hence, in this case we have $Q_{p\mu}(A+z)=Q_{p\mu}+2\sum_{i=1}^{\mu_1+r}Q_{p\mu/(i)}z^i$. So, we see that 
        \[R(z)=\sum_{p\leq \mu_1+r}\left(Q_\lambda\left(Q_{p\mu}+2\sum_{i=1}^{\mu_1+r}Q_{p\mu/(i)}z^i\right),Q_\nu(A-1/z)\right)\]
        is a Laurent polynomial of degree at most $|\lambda|(\mu_1+r)$.

        Next, we similarly have
        \[T(z)=\sum_{p>\mu_1+r} \left(Q_\lambda\left(Q_{p\mu}+2\sum_{i=1}^{\mu_1+r}Q_{p\mu/(i)}z^i\right),Q_\nu(A-1/z)\right).\]
        Note that $Q_\nu(A-1/z)$ expands into a linear combination of basis elements with weights at most $|\nu|$. Hence, when computing the inner product in $T(z)$, we only need to consider the case where $|\lambda|\cdot|p\mu/(i)|\leq|\nu|$, which is to say that $|\lambda|(|\mu|+p-i)\leq|\nu|$. Since $r$ is the greatest integer such that $|\lambda|(|\mu|+r)\leq|\nu|$, we only need to consider when $p-i\leq r$, i.e., $i\geq p-r$. 
        
        Therefore, we have
        \begin{align*}
            T(z)&=\sum_{p>\mu_1+r}\left(Q_\lambda\left(Q_{p\mu}+2\sum_{i\geq1}Q_{p\mu/(i)}z^i\right),Q_\nu(A-1/z)\right)\\
            &=\sum_{p>\mu_1+r}\left(Q_\lambda\left(2\sum_{i=p-r}^pQ_{p\mu/(i)}z^i\right),Q_\nu(A-1/z)\right)\\
            &=\sum_{p>\mu_1+r}\left(Q_\lambda\left(2\sum_{j=0}^rQ_{p\mu/(p-j)}z^{p-j}\right),Q_\nu(A-1/z)\right)
        \end{align*}
        after reindexing with $j=p-i$. Note that $p-j\geq p-r>(\mu_1+r)-r=\mu_1$, and so by Proposition \ref{p>t+r} we have that $Q_{p\mu/(p-j)}=q_jQ_\mu$. Thus, we have
        \begin{align*}
            T(z)&=\sum_{p>\mu_1+r}\left(Q_\lambda\left(2\sum_{j=0}^rq_jQ_\mu z^{p-j}\right),Q_\nu(A-1/z)\right)\\
            &=\sum_{p>\mu_1+r}z^{|\lambda|(p-r)}\left(Q_\lambda\left(2\sum_{j=0}^rq_jQ_\mu z^{r-j}\right),Q_\nu(A-1/z)\right)\\
            &=\frac{z^{|\lambda|(\mu_1+1)}}{1-z^{|\lambda|}}\cdot H(z)
        \end{align*}
        where $H(z)$ is a Laurent polynomial of degree at most $|\lambda|r$ of the form $H(z)=z^kH_2(z^{|\lambda|})$, for some $k\in\Z$. By Lemma \ref{separate}, we then have
        \begin{align*}
            T(z)&=G(z)+\frac{cz^k}{1-z^{|\lambda|}},
        \end{align*}
        where $c\in\Z$, and $G(z)$ is a Laurent polynomial of degree at most $|\lambda|(\mu_1+1+r)-|\lambda|=|\lambda|(\mu_1+r)$. Finally, we have
        \begin{align*}
            f(z)&=R(z)+\left(G(z)+\frac{cz^k}{1-z^{|\lambda|}}\right)\\
            &=L(z)+\frac{cz^k}{1-z^{|\lambda|}},
        \end{align*}
        where $L(z)$ is a Laurent polynomial of degree at most $|\lambda|(\mu_1+r)$.
    \end{proof}

    Next, we have a Schur's $Q$-function analogue of the stability property \cite[Theorem~4.2]{carre-thibon:1992}.
    
    \begin{theorem}\label{stability thm 2}
        Let $\lambda,\mu,\nu$ be partitions, and let $g(z)=\sum_{p,s\in\Z}\left(Q_{p\lambda}\of Q_\mu,Q_{s\nu}\right)z^s$.
        \begin{enumerate}
            \item If $\ell(\mu)>1$, then $g(z)$ is a Laurent polynomial of degree at most $\frac{|\nu|\cdot\mu_1}{|\mu|-\mu_1}$.
            \item If $\mu=(m)$, then
                \begin{align*}
                    g(z)&=P(z)+\frac{c_1z^k}{1-z^m}+\frac{c_2z^k}{(1-z^m)^2},
                \end{align*}
                where $c_1,c_2,k\in\Z$ and $P(z)$ is a Laurent polynomial of degree at most $(\lambda_1+|\nu|)m$.
        \end{enumerate}
    \end{theorem}
    
    \begin{proof}
        As in the proof of Theorem \ref{stability thm 1}, we have
        \[g(z)=\sum_{p\in\Z}\left(Q_{p\lambda} \left(Q_\mu+2\sum_{i=1}^{\mu_1}Q_{\mu/(i)}z^i\right),Q_\nu(A-1/z)\right).\]
        
        First, if $\ell(\mu)>1$, then $Q_{\mu/(i)}$ has weight $|\mu|-i\geq|\mu|-\mu_1\geq1$. Also, we have that the inner product $\left(Q_{p\lambda}\of Q_\mu(A+z),Q_\nu(A-1/z)\right)$ is $0$ when $(p+|\lambda|)(|\mu|-\mu_1)>|\nu|$, i.e., when $p>\frac{|\nu|}{|\mu|-\mu_1}-|\lambda|$. So, we see that $g(z)$ is a Laurent polynomial of degree at most $(p+|\lambda|)\cdot\mu_1=\frac{|\nu|\cdot\mu_1}{|\mu|-\mu_1}$.
        
        Next, suppose $\mu=(m)$, and let $r=|\nu|-|\lambda|$. Then, $g(z)=R(z)+T(z)$, where
        \[R(z)=\sum_{p\leq \lambda_1+r}\left(Q_{p\lambda}\of Q_{(m)}(A+z),Q_{s\nu}(A-1/z)\right)\]
        is a Laurent polynomial of degree at most $(|\lambda|+\lambda_1+r)m=(\lambda_1+|\nu|)m$, and
        \[T(z)=\sum_{p> \lambda_1+r}\left(Q_{p\lambda}\of Q_{(m)}(A+z),Q_{s\nu}(A-1/z)\right).\]
        Now, we can write $Q_{(m)}(A+z)=(Q_{(m)}(A+z)-z^m)+z^m$, so by the sum rule we have
        \[T(z)=\sum_{p> \lambda_1+r}\sum_{j=0}^p\left(Q_{p\lambda/(j)}(Q_{(m)}(A+z)-z^m)Q_{(j)}(z^m),Q_\nu(A-1/z)\right).\]
        Similarly, we write $Q_{(m)}(A+z)-z^m=(Q_{(m)}(A+z)-2z^m)+z^m$, and so by the skew sum rule we have
        \[T(z)=\sum_{p> \lambda_1+r}\sum_{j=0}^p\sum_{k=0}^p\left(Q_{p\lambda/(k)}(Q_{(m)}(A+z)-2z^m)Q_{(k-j)}(z^m)Q_{(j)}(z^m),Q_\nu(A-1/z)\right).\]
        Notice that
        \begin{align*}
            Q_{(m)}(A+z)-2z^m&=\sum_{i=0}^mq_{m-i}(A)q_i(z)-q_m(z)\\
            &=q_m(A)+2\sum_{i=1}^{m-1}q_{m-i}(A)z^i
        \end{align*}
        is a sum of terms with nonzero weight. Thus, in the inner product we only need to consider terms where $p+|\lambda|-k\leq|\nu|$, that is, where $k\geq p-r$. Additionally, $Q_{(k-j)}(z^m)=0$ unless $k-j\geq0$, i.e., $j\leq k$. Therefore, we have
        \begin{align*}
            T(z)&=\sum_{p> \lambda_1+r}\sum_{k=p-r}^p\left(Q_{p\lambda/(k)}(Q_{(m)}(A+z)-2z^m)\sum_{j=0}^kQ_{(k-j)}(z^m)Q_{(j)}(z^m),Q_\nu(A-1/z)\right)\\
            &=\sum_{p> \lambda_1+r}\sum_{i=0}^r\left(Q_{p\lambda/(p-i)}(Q_{(m)}(A+z)-2z^m)\sum_{j=0}^{p-i}Q_{(p-i-j)}(z^m)Q_{(j)}(z^m),Q_\nu(A-1/z)\right)
        \end{align*}
        after reindexing with $i=p-k$. Now, note that $p-i>(\lambda_1+r)-r=\lambda_1>0$. Hence, from the proof of \eqref{q_r(2z)} we have that
        \[\sum_{j=0}^{p-i}Q_{(p-i-j)}(z^m)Q_{(j)}(z^m)=q_{p-i}(2z^m)=4(p-i)z^{m(p-i)}.\]

        
        Furthermore, we can apply Proposition \ref{p>t+r} to get $Q_{p\lambda/(p-i)}=q_iQ_\lambda$ since $p>\lambda_1+r\geq\lambda_1+i$. Hence, we get
        \begin{align*}
            T(z)&=\sum_{p> \lambda_1+r}\sum_{i=0}^r\left(q_iQ_\lambda(Q_{(m)}(A+z)-2z^m)4(p-i)z^{m(p-i)},Q_\nu(A-1/z)\right)\\
            &=\sum_{p> \lambda_1+r}\sum_{i=0}^r4(p-i)z^{m(p-i)}\left(q_iQ_\lambda(Q_{(m)}(A+z)-2z^m),Q_\nu(A-1/z)\right).
        \end{align*}
        Now, we see that we may pull $z^{m(p-r)}$ out of the inner sum, and hence we have
        \begin{align*}
            T(z)&=\sum_{p> \lambda_1+r}z^{m(p-r)}\sum_{i=0}^r4(p-i)z^{m(r-i)}\left(q_iQ_\lambda(Q_{(m)}(A+z)-2z^m),Q_\nu(A-1/z)\right).
        \end{align*}
        Due to the factor $p-i$, we can write this as $T(z)=T_1(z)+T_2(z)$, where
        \begin{align*}
            T_1(z)&=\sum_{p> \lambda_1+r}pz^{m(p-r)}\sum_{i=0}^r4z^{m(r-i)}\left(q_iQ_\lambda(Q_{(m)}(A+z)-2z^m),Q_\nu(A-1/z)\right),
        \end{align*}
        and
        \begin{align*}
            T_2(z)&=\sum_{p> \lambda_1+r}z^{m(p-r)}\sum_{i=0}^r4(-i)z^{m(r-i)}\left(q_iQ_\lambda(Q_{(m)}(A+z)-2z^m),Q_\nu(A-1/z)\right).
        \end{align*}
        Next, we can see that we have $T_1(z)=\sum_{p> \lambda_1+r}pz^{m(p-r)}\cdot z^{k_1}H_1(z^m)$, where $z^{k_1}H_1(z^m)$ is a Laurent polynomial of degree at most $|\lambda|+r-1$. Furthermore, we see that
        \begin{align*}
            \sum_{p>0}pZ^p&=\frac{Z}{(1-Z)^2},
        \end{align*}
        and so we have
        \begin{align*}
            \sum_{p> \lambda_1+r}pz^{m(p-r)}&=z^{-mr}\sum_{p> \lambda_1+r}p(z^{m})^p\\
            &=\frac{z^{m(\lambda_1+r+1-r)}}{(1-z^m)^2}\\
            &=\frac{z^{m(\lambda_1+1)}}{(1-z^m)^2}.
        \end{align*}
        Therefore, this means we have
        \begin{align*}
            T_1(z)&=\frac{z^{m(\lambda_1+1)}}{(1-z^m)^2}\cdot z^{k_1}H_1(z^m).
        \end{align*}
        By Lemma \ref{separate}, we get
        \begin{align*}
            T_1(z)&=L_1(z)+\frac{c_1z^{k_1}}{1-z^m}+\frac{c_2z^{k_1}}{(1-z^m)^2}
        \end{align*}
        for some constants $c_1,c_2$, where $L_1(z)$ is a Laurent polynomial of degree at most $(|\lambda|+r-3)m$. Similarly, we have
        \begin{align*}
            T_2(z)&=\sum_{p> \lambda_1+r}z^{m(p-r)}\sum_{i=0}^r4(-i)z^{mi}\left(q_iQ_\lambda(Q_{(m)}(A+z)-2z^m),Q_\nu(A-1/z)\right)\\
            &=\frac{z^{m(\lambda_1+r+1-r)}}{1-z^m}\cdot z^{k_2}H_2(z^m)\\
            &=\frac{z^{m(\lambda_1+1)}}{1-z^m}\cdot z^{k_2}H_2(z^m)\\
            &=L_2(z)+\frac{c_3z^{k_2}}{1-z^m},
        \end{align*}
        where $L_2(z)$ is a Laurent polynomial of degree at most $(|\lambda|+r-1)m$. Therefore, we get
        \begin{align*}
            T(z)&=\left(L_1(z)+\frac{c_1z^{k_1}}{1-z^m}+\frac{c_2z^{k_1}}{(1-z^m)^2}\right)+\left(L_2(z)+\frac{c_3z^{k_2}}{1-z^m}\right)\\
            &=L(z)+\frac{c_1z^{k_1}+c_3z^{k_2}}{1-z^m}+\frac{c_2z^{k_1}}{(1-z^m)^2},
        \end{align*}
        where $L(z)$ has degree at most $(|\lambda|+r-3)m=(|\nu|-3)m$. Notice that the inner sums in the original definitions of $T_1(z)$ and $T_2(z)$ only differ by a factor of $-i$, so it follows that $k_1=k_2$. Finally, we see that $R(z)+L(z)$ has degree at most $(\lambda_1+|\nu|)m$.
    \end{proof}

    In other words, the stability theorems say that the sequences 
    \begin{align}
        (Q_\lambda\of Q_{p\mu},Q_{s\nu})&\qquad\text{for }p\in\Z,\text{ }s=|\lambda|(|\mu|+p)-|\nu|,\label{sequence:plethysm stability 1}\\
        (Q_{p\lambda}\of Q_{\mu},Q_{s\nu})&\qquad\text{for }p\in\Z,\text{ }s=(|\lambda|+p)|\mu|-|\nu|,~\ell(\mu)>1\label{sequence:plethysm stability 2}
    \end{align}
    stabilize for large enough $p$ (see Example \ref{example:first stability}). Meanwhile, when $\ell(\mu)=1$, the sequence (\ref{sequence:plethysm stability 2}) eventually increase linearly (see Example \ref{example:non-stability}). Recall from Proposition \ref{Q_(-p)lambda} that for $p<0$, we have $Q_{p\lambda}=0$ when $-p$ is not a part of $\lambda$. Since $\ell(\lambda)$ is finite, this means that both sequences always stabilize to $0$ as $p\to-\infty$.

    The difference between stability for Schur and Schur's $Q$-functions can be seen, in part, as follows. In our special case, we consider plethysm of the form $Q_{(p)}(Q_{(m)}(z))=4pz^{pm}$, and we see that the coefficient $4p$ increases linearly as $p\to\infty$. On the other hand, in the Schur function case we have $S_{(p)}(S_{(m)}(z))=z^{pm}$, which has a constant coefficient for all $p$.

    \begin{example}\label{example:first stability}
        Let $\lambda=(2,1)$, $\mu=(2)$, and $\nu=(4,3,2)$. We wish to compute the sequence $(Q_\lambda\of Q_{p\mu},Q_{s\nu})$ for $p\in\Z$. First, recall that for any $F\in\Gamma$, we may write
        \begin{align*}
            F&=\sum_\gamma \frac{(F,Q_\gamma)}{(Q_\gamma,Q_\gamma)}Q_\gamma=\sum_\gamma 2^{-\ell(\gamma)}(F,Q_\gamma)Q_\gamma,
        \end{align*}
        where the sum is over strict partitions $\gamma$. For $p\geq2$ we compute via computer algebra the basis expansion of $Q_\lambda\of Q_{p\mu}$,
        \begin{align*}
            Q_\lambda\of Q_{2\mu} &=0\\
            Q_\lambda\of Q_{3\mu} &=60\cdot Q_{6\nu}+\cdots,\\
            Q_\lambda\of Q_{4\mu} &=536\cdot Q_{9\nu}+\cdots,\\
            Q_\lambda\of Q_{5\mu} &=664\cdot Q_{12\nu}+\cdots,\\
            Q_\lambda\of Q_{6\mu} &=664\cdot Q_{15\nu}+\cdots,\\
            Q_\lambda\of Q_{7\mu} &=664\cdot Q_{18\nu}+\cdots.
        \end{align*}
        After multiplying these coefficients by $2^4$, we get the sequence
        \[0,\qquad 960,\qquad 8576,\qquad 10624,\qquad 10624,\qquad 10624,\qquad \cdots\]
        which stabilizes to $10624$.
    \end{example}

    \begin{example}
        Let $\lambda=(1)$, $\mu=(2,1)$, and $\nu=(3,2)$. We note that $\ell(\mu)>1$. We wish to compute the sequence $(Q_{p\lambda}\of Q_{\mu},Q_{s\nu})$ for $p\geq1$. We compute
        \begin{align*}
            Q_{1\lambda}\of Q_\mu &=0\cdot Q_{1\nu}+\cdots,\\
            Q_{2\lambda}\of Q_\mu &=12\cdot Q_{4\nu}+\cdots,\\
            Q_{3\lambda}\of Q_\mu &=168\cdot Q_{7\nu}+\cdots,\\
            Q_{4\lambda}\of Q_\mu &=204\cdot Q_{10\nu}+\cdots,\\
            Q_{5\lambda}\of Q_\mu &=0\cdot Q_{13\nu}+\cdots,\\
            Q_{6\lambda}\of Q_\mu &=0\cdot Q_{16\nu}+\cdots,\\
            Q_{7\lambda}\of Q_\mu &=0\cdot Q_{19\nu}+\cdots.
        \end{align*}
        So, we get the sequence $0,96,1344,1632,0,0,0,\ldots$, which stabilizes to $0$.
    \end{example}

    \begin{example}\label{example:non-stability}
        Let $\lambda=(1)$, $\mu=(3)$, and $\nu=(2,1)$. We note that $\ell(\mu)=1$. We wish to compute the sequence $(Q_{p\lambda}\of Q_{\mu},Q_{s\nu})$ for $p\geq1$. We compute
        \begin{align*}
            Q_{1\lambda}\of Q_\mu &=0\cdot Q_{3\nu}+\cdots,\\
            Q_{2\lambda}\of Q_\mu &=12\cdot Q_{6\nu}+\cdots,\\
            Q_{3\lambda}\of Q_\mu &=88\cdot Q_{9\nu}+\cdots,\\
            Q_{4\lambda}\of Q_\mu &=256\cdot Q_{12\nu}+\cdots,\\
            Q_{5\lambda}\of Q_\mu &=464\cdot Q_{15\nu}+\cdots,\\
            Q_{6\lambda}\of Q_\mu &=672\cdot Q_{18\nu}+\cdots,\\
            Q_{7\lambda}\of Q_\mu &=880\cdot Q_{21\nu}+\cdots.
        \end{align*}
        So, we get the sequence $0,96,704,2048,3712,5376,7040,\ldots$. This sequence does not stabilize, but the differences between terms is the sequence $0,96,608,1344,1664,1664,1664,\ldots$, which stabilizes to $1664$.
    \end{example}
    
    \section{Recurrence Formulas}
	
    Now, we show that Schur's $Q$-functions have recurrence formulas that are similar to those for Schur functions. Recurrence formulas for Schur functions have historically been used to assist with difficult plethysm computations. In \cite{carre-thibon:1992}, it is shown that vertex operator methods may be used to generalize recurrence relations from \cite{murnaghan:1954} and \cite{butler-king:1973}. We now show that the vertex operator method also produces analogous recurrence relations for Schur's $Q$-functions. However, in this case, the recurrence relations have many more terms.

    In particular, the Foulkes conjecture \cite{foulkes:1950} states that $S_n\of S_m-S_m\of S_n$ is Schur-positive when $m\leq n$. Considering the Schur's $Q$-function analogue of this statement, plethysm of the form $q_n\of q_m$ is of interest. We are able to find an analogue of the Butler-King formula to compute $D_k(q_n\of q_m)$, where $D_k$ is defined in \eqref{def D_k}.

    \subsection{The Recurrence Formula}


    We start by computing the action of the vertex operator $\kappa_z \cdot\kappa_{-1/z}^\perp$ on $\kappa_1(q_m(A))$.

    \begin{lemma}\label{as_prod}
        Let $m>0$ be a positive integer, then
        \[\kappa_z \cdot\kappa_{-1/z}^\perp\kappa_1(q_m(A))=\kappa_z(A)\kappa_1(q_m)\left(\prod_{n=1}^m \kappa_{(-1/z)^n}(q_{m-n})\right)^2.\]
    \end{lemma}
    
    \begin{proof}
        First, we expand out the LHS. We apply Lemma \ref{F(A+z)} to get
        \[\kappa_z \kappa_{-1/z}^\perp\kappa_1(q_m(A))=\kappa_z(A)\kappa_1(q_m(A-1/z)),\]
        and then we use the sum rule to get
        \[\kappa_z(A)\kappa_1\left(\sum_{i=0}^m q_{m-i}(A)q_i(-1/z)\right).\]
        From \eqref{q_r(-z)}, we get
        \[\kappa_z(A)\kappa_1\left(q_m(A)+2\sum_{i=1}^m (-1/z)^i q_{m-i}(A)\right).\]
        Now, recall from (\ref{eqn:kappa_z(A+B)}) that $\kappa_z(A+B)=\kappa_z(A)\kappa_z(B)$. Therefore, we have
        \begin{align*}
            \kappa_z(A)\kappa_1(q_m)\left(\kappa_1\left(\sum_{i=1}^m (-1/z)^i q_{m-i}(A)\right)\right)^2.
        \end{align*}
        It remains to show that
        \begin{equation}\label{eqn:product of gen fcns}
            \kappa_1\left(\sum_{i=1}^m (-1/z)^i q_{m-i}(A)\right)=\prod_{n=1}^m \kappa_{(-1/z)^n}(q_{m-n}).
        \end{equation}
        First, from Proposition \ref{plethysm zA} we have $q_n(zA)=z^nq_n(A)$, and so
        \begin{align*}
            \kappa_1(zA)&=\sum_{n\in\Z}q_n(zA)\\
            &=\sum_{n\in\Z}z^nq_n(A)\\
            &=\kappa_z(A).
        \end{align*}
        It follows that $\kappa_1((-1/z)^iq_{m-i})=\kappa_{(-1/z)^i}(q_{m-i})$, and so by applying \eqref{eqn:kappa_z(A+B)} we have \eqref{eqn:product of gen fcns}.

    \end{proof}

    Carr\'e and Thibon \cite{carre-thibon:1992} showed that the Schur functions satisfy the analogous identity
    \[\sigma_z\lambda_{-1/z}^\perp \sigma_1(S_m(A))=\sigma_z(A)\sigma_1(S_m)\lambda_{-1/z}(S_{m-1}),\]
    where $\sigma_z$ and $\lambda_z$ are the generating functions of the elementary and homogeneous symmetric functions defined in \eqref{eqn:sigma_z} and \eqref{eqn:lambda_z}, respectively. Hence, we see that with Schur's $Q$-functions we get a product whose number of factors depends on $m$, whereas there are exactly three factors for Schur functions. We now use Lemma \ref{as_prod} to obtain the recurrence formula. First, we will make use of some operators that act on $Q_\lambda$ by removing a part of $\lambda$. For positive integers $k>0$, we define the operator $D_k$ to act on $Q_\lambda$ by
    \begin{equation}\label{def D_k}
        D_k(Q_\lambda):=
        \begin{cases}
            (-1)^{i+1}2Q_{\lambda\setminus\{\lambda_i\}}&\text{if }k=\lambda_i\text{ for some i},\\
            0&\text{otherwise}.
        \end{cases}
    \end{equation}
    It follows from Proposition \ref{Q_(-p)lambda} that we get the following.
    \begin{remark}\label{Q_(-k)lam}
        For a partition $\lambda$ and positive integer $k>0$, we have
        \[Q_{(-k)\lambda}=(-1)^kD_k Q_\lambda.\]
    \end{remark}

    Now, we are ready to prove a Schur's $Q$-function analogue of the recurrence formula \cite[Theorem 5.1]{carre-thibon:1992}.
    
    \begin{theorem}[Recurrence Formula]\label{thm_5_1_analogue}
        Let $m>0$ be a positive integer, then
        \[D_k \kappa_1(q_m(A))=\kappa_1(q_m)\sum_{(I+J)\cdot\delta-r=k}(-1)^rq_r(A)\prod_{s=1}^m q_{i_s}(q_{m-s})q_{j_s}(q_{m-s}),\]
        where $I=(i_1,\ldots,i_m),J=(j_1,\ldots,j_m)\in\Z^m$, $\delta=(1,2,\ldots,m)$, and $I\cdot\delta$ is the usual dot product.
    \end{theorem}
    
    \begin{proof}
        First, we start with the product $\kappa_z \cdot \kappa_{-1/z}^\perp\kappa_1(q_m(A))$, and expand it in a different way then in Lemma \ref{as_prod}. Let $\kappa_1(q_m(A))=\sum_\lambda b_\lambda Q_\lambda$ be the expansion of $\kappa_1(q_m(A))$ as a linear combination of basis elements. Substituting this expansion for $\kappa_1(q_m(A))$, we get
        \begin{align*}
            \kappa_z \kappa_{-1/z}^\perp\kappa_1(q_m(A))&=\kappa_z \kappa_{-1/z}^\perp\sum_\lambda b_\lambda Q_\lambda\\
            &=\sum_\lambda b_\lambda \kappa_z \kappa_{-1/z}^\perp Q_\lambda.
        \end{align*}
        Then, we apply Theorem \ref{vertex operator identity} to $\kappa_z \kappa_{-1/z}^\perp Q_\lambda$ to get $\sum_\lambda b_\lambda\sum_{p\in\Z} Q_{p\lambda}z^p$. Combining this with Lemma \ref{as_prod}, we have
        \begin{equation}\label{sum_=_prod}
            \sum_{p\in\Z}\sum_\lambda z^p b_\lambda Q_{p\lambda}(A)=\kappa_z(A)\kappa_1(q_m)\left(\prod_{n=1}^m \kappa_{(-1/z)^n}(q_{m-n})\right)^2.
        \end{equation}

        On the LHS of \eqref{sum_=_prod}, we see that the coefficient of $z^{-k}$ is
        \begin{align*}
            \sum_\lambda b_\lambda Q_{(-k)\lambda}&=\sum_\lambda b_\lambda(-1)^kD_k Q_\lambda\\
            &=(-1)^kD_k\sum_\lambda b_\lambda Q_\lambda\\
            &=(-1)^kD_k\kappa_1(q_m).
        \end{align*}
        Meanwhile, on the RHS of \eqref{sum_=_prod} we see that the coefficient of $z^{-k}$ is
        \begin{align*}
            \sum&(-1)^{1i_1+\cdots+mi_m+1j_1+\cdots+mj_m}q_r(A)\kappa_1(q_m)q_{i_1}(q_{m-1})\cdots q_{i_m}(q_0)q_{j_1}(q_{m-1})\cdots q_{j_m}(q_0),
        \end{align*}
        where the sum ranges over all $I,J\in\Z^m$ and $r\in\Z$ such that $r-1i_1-\cdots-mi_m-1j_1+\cdots+mj_m=-k$, i.e., such that $r-(I+J)\cdot\delta=-k$. Since $(I+J)\cdot\delta=k+r$, we may rewrite this as
        \begin{align*}
            \kappa_1(q_m)\sum_{(I+J)\cdot\delta-r=k}(-1)^{k+r}q_r(A)q_{i_1}(q_{m-1})\cdots q_{i_m}(q_0)q_{j_1}(q_{m-1})\cdots q_{j_m}(q_0).
        \end{align*}
    \end{proof}
    


    \subsection{Specializations of the Recurrence Formula}

    In \cite{carre-thibon:1992}, they specialize their recurrence formula for Schur functions to get the reccurence formulas of Butler-King and Murnaghan. Define the operators acting on Schur functions by $C_k(S_\lambda)=S_{\lambda-(1^k)}$ if $\ell(\lambda)=k$ and $C_k(S_\lambda)=0$ otherwise. The Butler-King formulas \cite{butler-king:1973} are
    \begin{align*}
        C_k(S_n\of S_m)&=\sum_{j=0}^{n-k}(-1)^j S_{n-k-j}(S_m) e_{k+j}(S_{m-1})S_j,\\
        C_k(e_n\of S_m)&=\sum_{j=0}^{n-k}(-1)^je_{n-k-j}(S_m)S_{k+j}(S_{m-1})S_j.
    \end{align*}
    Murnaghan's formulas \cite{murnaghan:1954} are
    \begin{align*}
        \sum_{p=0}^{n-k}(-1)^p e_p(S_m)\cdot C_k S_{n-p}(S_m)=(-1)^{n-k}e_n(S_{m-1})S_{n-k},\\
        \sum_{p=0}^{n-k}(-1)^pS_p(S_m)\cdot C_k e_{n-p}(S_m)=(-1)^{n-k}S_n(S_{m-1})S_{j-k}.
    \end{align*}

    The B-K equations above are related to each other through the involution $\omega$. Since Schur's $Q$-functions are invariant under $\omega$, we only get one B-K analogue, and similarly we only get one Murnaghan analogue. To find these analogues for Schur's $Q$-functions, we can identify terms of equal weights on both sides of the equation in Theorem \ref{thm_5_1_analogue}. First, we have the analogue of the B-K formulas.
    
    \begin{theorem}\label{B-K_analogue}
        Let $k,n,m>0$ be positive integers, then
        \[D_k(q_n\of q_m)=\sum_{I,J\in\Z^m}(-1)^{(I+J)\cdot\delta-k}q_{(I+J)\cdot\delta-k}(A)q_{n-|I|-|J|}(q_m)\prod_{s=1}^m q_{i_s}(q_{m-s})q_{j_s}(q_{m-s}).\]
    \end{theorem}
    
    \begin{proof}
        We identify the terms of equal weights ($nm-k$) from Theorem \ref{thm_5_1_analogue}. On the LHS, this is just $D_k(q_n\of q_m)$. We see that the RHS is multiplying $\sum_{\ell\geq0}q_\ell(q_m)$ with terms of weight $r+\sum_{s=1}^m i_s(m-s)+\sum_{s=1}^m j_s(m-s)=r+\sum_{s=1}^m(i_s+j_s)(m-s)$. Thus, we have weight
        \begin{align*}
            \ell m+r+\sum_{s=1}^m (i_s+j_s)(m-s)&=\ell m+r+m\sum_{s=1}^m (i_s+j_s) -\sum_{s=1}^m s(i_s+j_s)\\
            &=\ell m+r+m(|I|+|J|)-(I+J)\cdot\delta\\
            &=(\ell+|I|+|J|)m+r-(k+r)\\
            &=(\ell+|I|+|J|)m-k.
        \end{align*}
        Since $nm-k=(\ell+|I|+|J|)m-k$, we see that $\ell=n-|I|-|J|$. Also, note that $r=(I+J)\cdot\delta-k$. Thus, we get
        \begin{align*}
            D_k(q_n\of q_m)&=\sum_{I,J\in\Z^m} (-1)^r q_\ell(q_m)q_r(A)\prod_{s=1}^m q_{i_s}(q_{m-s})q_{j_s}(q_{m-s})\\
            &=\sum_{I,J\in\Z^m}(-1)^{(I+J)\cdot\delta-k}q_{n-|I|-|J|}(q_m)q_{(I+J)\cdot\delta-k}(A)\prod_{s=1}^m q_{i_s}(q_{m-s})q_{j_s}(q_{m-s}).
        \end{align*}
        Note that there is no restriction on $I,J\in\Z^m$ in the sum since $(I+J)\cdot\delta-r=(I+J)\cdot\delta-((I+J)\cdot\delta-k)=k$ for all $I,J$. However, it may be convenient to restrict to $I,J\in\Z_{\geq0}^m$ with $|I|+|J|\leq n$ and $(I+J)\cdot\delta\geq k$ since other terms will be $0$.
    \end{proof}

    Now, we may do a different specialization to obtain a Schur's $Q$-function analogue of Murnaghan's formulas.
    
    
    
    \begin{theorem}\label{Murn_analogue}
        Let $k,n,m>0$ be positive integers, then
        \[\sum_{p=0}^{n-k}(-1)^p q_p(q_m)\cdot D_k q_{n-p}(q_m)=\sum_{\substack{I,J\in\Z^m\\|I|+|J|=n}}(-1)^{((I+J)\cdot\delta-k}q_{(I+J)\cdot\delta-k}(A)\prod_{s=1}^m q_{i_s}(q_{m-s})q_{j_s}(q_{m-s}).\]
    \end{theorem}
    
    \begin{proof}
        First, note that $1/\kappa_1(q_m)=\kappa_{-1}(q_m)$. So, from Theorem \ref{thm_5_1_analogue} we have
        \begin{equation}\label{5.4_analogue}
            \kappa_{-1}(q_m)\cdot D_k \kappa_1(q_m)=\sum_{I,J\in\Z^m}(-1)^{(I+J)\cdot\delta-k}q_{(I+J)\cdot\delta-k}(A)\prod_{s=1}^m q_{i_s}(q_{m-s})q_{j_s}(q_{m-s}).
        \end{equation}
    
        We identify the terms of weight $nm-k$ in \eqref{5.4_analogue}. We see that
        \begin{align*}
            (I+J)\cdot\delta-k+\sum_{s=1}^m (i_s+j_s)(m-s)&=(I+J)\cdot\delta-k+m\sum_{s=1}^m (i_s+j_s)-\sum_{s=1}^m s(i_s+j_s)\\
            &=(I+J)\cdot\delta-k+m(|I|+|J|)-(I+J)\cdot\delta\\
            &=m(|I|+|J|)-k,
        \end{align*}
        and so in the sum on the RHS we require $|I|+|J|=n$.
    \end{proof}

    In practice, the plethysm $F\of G$ is easiest to compute by expanding $F$ and $G$ in the power sum basis, as described in \cite{loehr-remmel}. Thus, using the power sum expansions of the $q_n$'s \cite[p.~260]{macdonald:1995}, we compute $q_n\of q_m$ for small values of $n$ and $m$.

    \begin{remark}\label{remark:q_n(q_m)}
        We have
        \begin{align*}
            q_0(q_n)&=1,\\
            q_n(q_0)&=2, \quad(n\geq1)\\
            q_1(q_n)&=2q_n,\\
            q_2(q_n)&=2q_n^2\\
            &=4Q_{(2n)}+4Q_{(2n-1,1)}+4Q_{(2n-2,n)}+\cdots+4Q_{(n+1,n-1)},\\
            q_3(q_1)&=2q_3+q_1^3\\
            &=6Q_{({3})}+2Q_{({2, 1})},\\
            q_3(q_2)&=2q_3q_2q_1+2q_4q_2-6q_5q_1+6q_6\\
            &=6Q_{({6})}+10Q_{({5, 1})}+14Q_{({4, 2})}+2Q_{({3, 2, 1})},\\
            q_3(q_3)&=2q_2q_3q_4+2q_1q_3q_5-6q_1q_2q_6-6q_4q_5+8q_3q_6+10q_2q_7-6q_1q_8+2q_9\\
            &=6Q_{({9})}+10Q_{({8, 1})}+22Q_{({7, 2})}+24Q_{({6, 3})}+10Q_{({6, 2, 1})}+10Q_{({5, 4})}+14Q_{({5, 3, 1})}+2Q_{({4, 3, 2})}.
        \end{align*}
    \end{remark}

    \begin{example}\label{ex1}
        We wish to verify Theorem \ref{B-K_analogue} for $D_1(q_3\of q_2)$. First, we will compute this directly. By Remark \ref{remark:q_n(q_m)}, we have
        \begin{align*}
            q_3\of q_2&=2\,q_{1}q_{2}q_{3}+2\,q_{2}q_{4}-6\,q_{1}q_{5}+6\,q_{6}\\
            &=6Q_{({6})}+10Q_{({5, 1})}+14Q_{({4, 2})}+2Q_{({3, 2, 1})}.
        \end{align*}
        Therefore, by the definition of $D_1$ we have
        \begin{align*}
            D_1(q_3\of q_2)&=-20Q_{({5})}+4Q_{({3, 2})}.
        \end{align*}

        On the other hand, there are 34 pairs of $I,J\in\Z^2$ that contribute a nonzero term in Theorem \ref{B-K_analogue}. A select few of these are as follows:\\
        
        \begin{center}
        \begin{tabular}{c|c|c}
            $I$ & $J$ & Term (after simplification) \\\hline
            $(0,0)$ & $(0,1)$ & $-16q_2q_3+8q_1q_4$ \\
            $(0,0)$ & $(0,2)$ & $-4q_2q_3$ \\
            $(0,1)$ & $(0,0)$ & $-16q_2q_3+8q_1q_4$ \\
            $(1,1)$ & $(0,1)$ & $8q_1q_4$ \\
            $(2,0)$ & $(0,1)$ & $-8q_2q_3$
        \end{tabular}
        \end{center}
        \vspace{1em}
        After summing all 34 terms, we are again left with 
        \[2\,q_{1}q_{2}q_{3}+2\,q_{2}q_{4}-6\,q_{1}q_{5}+6\,q_{6}.\]




    \end{example}

    \subsection{Maximal First Part}

    Similarly to Carr\'e and Thibon, we now consider the plethysm $q_n\of Q_\lambda$, expanded as a sum of the form $\sum_\mu d_\mu Q_\mu$. We wish to find a formula for the nonzero terms in this expansion where $\mu$ has the largest possible first part. In \cite{carre-thibon:1992}, Carr\'e and Thibon found analogous formulas for Schur functions for the terms with maximal first part, as well as for terms with maximal length.
    
    \begin{proposition}\label{max first part}
        Let $\lambda$ be a partition and $n>0$ be a positive integer, then
        \[D_{n\lambda_1}(q_n\of Q_\lambda)=q_n\of (2Q_{\lambda/(\lambda_1)}).\]
    \end{proposition}
    
    \begin{proof}
        Suppose the expansion of $q_n\of Q_\lambda$ is $q_n\of Q_\lambda=\sum_\mu d_\mu Q_\mu$. We apply the operator $\kappa_z\kappa_{-1/z}^\perp$ to both sides of this expansion. On the RHS, we get $\sum_\mu\sum_{p\in\Z}d_\mu Q_{p\mu}z^p$. On the LHS, we get
        \[\kappa_z\kappa_{-1/z}^\perp\sum_\mu d_\mu Q_\mu=\kappa_z(A)q_n(Q_\lambda(A-1/z))\]
        by applying Theorem \ref{vertex operator identity}. Then, we can use the sum rule to get
        \begin{align*}
            &\kappa_z(A)q_n\left(\sum_\nu Q_{\lambda/\nu}(A)Q_\nu(-1/z)\right)\\
            &=\kappa_z(A)q_n\left(\sum_k Q_{\lambda/(k)}(A)q_k(-1/z)\right)\\
            &=\kappa_z(A)q_n\left(Q_\lambda(A)+2\sum_{k\geq1}(-1/z)^k Q_{\lambda/(k)}(A)\right)
        \end{align*}
        since $Q_\nu(-1/z)$ is $0$ if $\ell(\nu)>1$. Therefore, we have
        \begin{equation}\label{maximal intermediate eqn}
            \kappa_z(A)q_n\left(Q_\lambda(A)+2\sum_{k\geq1}(-1/z)^k Q_{\lambda/(k)}(A)\right)=\sum_\mu\sum_{p\in\Z}d_\mu Q_{p\mu}z^p.
        \end{equation}
        Note that $Q_{\lambda/(k)}=0$ for $k>\lambda_1$, so the lowest power of $z$ on the LHS of \eqref{maximal intermediate eqn} is $z^{-n\lambda_1}$, and its coefficient is
        \begin{align*}
            q_n(2(-1)^{\lambda_1}Q_{\lambda/(\lambda_1)}(A))&=(-1)^{n\lambda_1}q_n(2Q_{\lambda/(\lambda_1)}).
        \end{align*}
        Let $h=n\lambda_1$, then on the RHS of \eqref{maximal intermediate eqn}, the coefficient of $z^{-h}$ is 
        \begin{align*}
            \sum_\mu d_\mu Q_{(-h)\mu}&=\sum_\mu d_\mu(-1)^h D_h Q_\mu\\
            &=(-1)^hD_h\sum_\mu d_\mu Q_\mu\\
            &=(-1)^hD_h(d_n\of Q_\lambda).
        \end{align*}
        Therefore, equating coefficients we get
        \[(-1)^hD_h(d_n\of Q_\lambda)=(-1)^hq_n(2Q_{\lambda/(\lambda_1)}).\]
    \end{proof}




    \begin{example}
        We have that $Q_{\lambda/(\lambda_1)}=Q_{\lambda\setminus\{\lambda_1\}}$, and so by Proposition \ref{max first part} we have
        \[D_{n\lambda_1}(q_n\of Q_\lambda)=q_n\of 2Q_{\lambda\setminus\{\lambda_1\}}.\]
        In particular, if $n=1$, then we see that we have
        \[D_{\lambda_1}(q_1\of Q_\lambda)=q_1\of 2Q_{\lambda\setminus\{\lambda_1\}}=4Q_{\lambda\setminus\{\lambda_1\}}.\]
        Computing the action of $D_{\lambda_1}$ directly with \eqref{def D_k}, we also get
        \[D_{\lambda_1}(q_1\of Q_\lambda)=D_{\lambda_1}(2Q_\lambda)=4(-1)^2Q_{\lambda\setminus\{\lambda_1\}}.\]
    \end{example}

    \appendix

    \section{Schur Functions and the Ring $\Lambda$}\label{appendix:schur}

    Let
    \begin{equation}\label{eqn:sigma_z}
        \sigma_z(A):=\prod_{a\in A}\frac{1}{1-za}=\sum_{n\in\Z}h_n(A)z^n
    \end{equation}
    and
    \begin{equation}\label{eqn:lambda_z}
        \lambda_z(A):=\prod_{a\in A}(1+za)=\sum_{n\in\Z}e_n(A)z^n
    \end{equation}
    be the generating functions of the homogeneous symmetric functions $h_n$ and elementary symmetric functions $e_n$, respectively. The ring $\Lambda$ of symmetric functions in the variables of $A$ is defined
    \[\Lambda:=\Q[h_1,h_2,h_3,\ldots]=\Q[e_1,e_2,e_3,\ldots].\]
    For two compositions $\lambda,\mu$, we define the skew Schur function $S_{\lambda/\mu}$ by
    \[S_{\lambda/\mu}:=\det(h_{\lambda_i-\mu_j-i+j}),\]
    and the Schur function $S_\lambda:=S_{\lambda/0}$. We define $\omega:\Lambda\to\Lambda$ by $\omega(h_n)=e_n$ for all $n\geq0$, and it follows that $\omega$ is an involution, i.e., $\omega(e_n)=h_n$. 

    We use the $\lambda$-ring definition of plethysm from \cite{lascoux:2003} to define plethysm on $\Lambda$. For $P=\sum_\mu c_\mu x^\mu\in\C[X]$, we define the plethysm $h_n(P)$ by
    \begin{equation}
        \sigma_z(P)=\prod_{a\in A}\left(\frac{1}{1-zx^\mu}\right)^{c_\mu}=\sum_{n\in\Z}h_n(P)z^n.
    \end{equation}
    For any polynomial in the $h_n(A)$'s, say $F(A)=\cF(h_1(A),h_2(A),\ldots)$. So, we define plethysm of $F$ on $P$ as
    \[F(P):=\cF(h_1(P),h_2(P),\ldots).\]
    Equivalently, we can define $e_n(P)$ by 
    \begin{equation}
        \lambda_z(P)=\prod_{a\in A}(1+zx^\mu)^{c_\mu}=\sum_{n\in\Z}e_n(P)z^n
    \end{equation}
    since $\sigma_z=\lambda_{-z}^{-1}$. Then, because $\kappa_z(P)=\sigma_z(P)\lambda_z(P)$, we have the following.

    \begin{remark}
        Our definition (\ref{eqn:kappa_z(P)}) of plethysm in $\Gamma$ is equivalent to the definition of plethysm in the ring $\Lambda$ of all symmetric functions when restricted to $\Gamma$.
    \end{remark}

    \section{Combinatorial Interpretations of Schur's $Q$-functions}\label{appendix:combinatorial}

    First, we note that for $n\geq1$ we have
    \[q_n=\sum_{\mu}2^{\ell(\mu)}m_\mu,\]
    where $m_\mu$ is the \emph{monomial symmetric function} indexed by $\mu$, and the sum ranges over partitions such that $|\mu|=n$. In terms of the Schur functions, we get
    \[q_n=S_{(n)}+S_{(n-1,1)}+\cdots+S_{(1^n)}.\]

    Next, we have
    \[\kappa_z=\prod_{i\geq 1}(1+2a_iz+2(a_iz)^2+2(a_iz)^3+\cdots),\]
    and so the two summands of $2(a_iz)^r$ corresponds to the semistandard Young tableaux 
    \[
        \begin{ytableau}
            $i'$ & $i$ & ~~\cdots~~ & $i$
        \end{ytableau}
        \qquad\text{and}\qquad
        \begin{ytableau}
            $i$ & $i$ & ~~\cdots~~ & $i$
        \end{ytableau}
    \]
    of shape $(r)$. It follows that $q_n$ enumerates the the semistandard Young tableaux of shape $(n)$ with entries in the ordered alphabet $1<1'<2<2'<\cdots$ such that each $k'$ appears at most once. The resulting theory of shifted Young tableaux provides a combinatorial construction of $Q_\lambda$ \cite{sagan:1987}.

    \section*{Acknowledgment}

    We thank the anonymous referee for the valuable insights and feedback.

    \section*{Declarations}

    \noindent\textbf{Ethical Approval:} Not applicable.

    \noindent\textbf{Funding:} Partially supported by Simons Foundation grant MP-TSM-00002518 and NSFC 12171303.

    \noindent\textbf{Availability of data and materials:} Not applicable.

    \bibliographystyle{alpha}
    \bibliography{references}

    \bigskip
    \bigskip

    
    
    
    

\end{document}